\documentclass[10pt, reqno]{amsart}
\usepackage[utf8]{inputenc}

\usepackage{style}
\newcommand{\chiDP}{\chi_f^{\mathrm{DP}}}

\usepackage[backend=biber, sorting=nyt, maxnames=100,backref=true, giveninits=true, sortcites=true]{biblatex}
\title[Fractional DP-colorings of $d$-degenerate locally sparse graphs]{Fractional DP-colorings of $d$-degenerate\\
locally sparse graphs}

\author{Abhishek Dhawan$^{\star}$}
\address{$^{\star}$Department of Mathematics, University of Illinois Urbana--Champaign.}
\email{adhawan2@illinois.edu}
\thanks{Abhishek Dhawan is partially supported by the NSF RTG grant DMS-1937241.}

\author{Huy Nguyen$^{\star}$}
\email{huyn4@illinois.edu}

\author{Rohan A. Rathi$^{\dagger}$}
\address{$^\dagger$Independent researcher.}
\email{rohanrathi135@gmail.com}

\date{}

\begin{document}

\begin{abstract}
    Bernshteyn, Kostochka, and Zhu (2020) introduced the notion of fractional DP-coloring, which generalizes both fractional coloring and fractional list coloring. Among several foundational results, they proved that every $d$-degenerate bipartite graph $G$ satisfies $\chiDP(G) \le (1 + o(1))\frac{d}{\log d}$, and that this bound is optimal---a stark contrast to ordinary fractional coloring. In this paper, we extend this upper bound to all $d$-degenerate triangle-free graphs, proving that $\chiDP(G) \le (4 + o(1))\frac{d}{\log d}$. This generalizes a recent result of Martinsson and Steiner (2025) for ordinary fractional coloring.

    We derive this result as a corollary of a more general upper bound concerning locally sparse graph orderings. Specifically, a $d$-degenerate graph $G$ is \emph{left $k$-locally-sparse} if it admits a degeneracy ordering in which, for every vertex $v$, the subgraph induced by its back-neighbors contains at most $k$ edges. We show that if a $d$-degenerate graph $G$ is left $\frac{d^2}{f}$-locally-sparse, then
    \[
        \chiDP(G) \le (8 + o(1))\frac{d}{\log f}.
    \]
    This immediately yields an identical upper bound on the ordinary fractional chromatic number $\chi_f(G)$, improving upon the leading constants of previously known bounds.
    
    Additionally, we establish the asymptotic sharpness of this result up to the leading constant. For any $1 \ll f \le d^2$, we construct $d$-degenerate graphs that are left $\frac{d^2}{f}$-locally-sparse and satisfy $\chi_f(G) \ge (1 - o(1))\frac{d}{\log f}$.
    
    Finally, as applications of our main theorem, we obtain improved upper bounds on the fractional DP-chromatic number of $d$-degenerate $K_{1,t,t}$-free graphs, as well as $K_{t,t,t}$-free graphs with maximum degree $\Delta$. Notably, these bounds improve upon existing results even in the setting of ordinary fractional coloring.
\end{abstract}

\maketitle

\sloppy

\section{Introduction}\label{section: intro}

\subsection{Background}\label{subsection: background}

All graphs considered in this work are finite, simple, and undirected. A \emph{coloring} of a graph $G = (V, E)$ is a mapping $\phi\,:\,V(G) \to \N$, which is \emph{proper} if $\phi(x) \neq \phi(y)$ for every edge $xy \in E(G)$. The \emph{chromatic number} $\chi(G)$ is the minimum size of a set $C \subseteq \N$ for which $G$ admits a proper coloring $\phi\,:\,V(G) \to C$. Bounding $\chi(G)$ across various graph classes is a central problem in graph theory. In this paper, we focus on the \emph{fractional chromatic number}, denoted $\chi_f(G)$, of $d$-degenerate graphs---graphs that admit a vertex ordering in which each vertex has at most $d$ back-neighbors. Here, we extend several classical bounds on graph coloring under local sparsity conditions to this setting.

For $k \in \mathbb{R}$, we say $G$ is \emph{$k$-locally-sparse} if for each $v \in V(G)$, the induced subgraph $G[N(v)]$ contains at most $\lfloor k \rfloor$ edges. (Allowing $k \in \mathbb{R}$ rather than $k \in \N$ avoids taking floors and simplifies our notation.) Molloy showed that local sparsity can be leveraged to refine the standard greedy bound $\chi(G) \le \Delta(G) + 1$:

\begin{theorem}[{Molloy~\cite{molloy1997bound}}]\label{theo:MR}
    For every $\delta > 0$, there exists $\eps > 0$ such that for all sufficiently large $\Delta \in \N$ (in terms of $\delta$ and $\eps$), if $G$ is a $(1-\delta)\binom{\Delta}{2}$-locally-sparse graph of maximum degree $\Delta$, then $\chi(G) \le (1 - \eps)(\Delta + 1)$.
\end{theorem}

Alon, Krivelevich, and Sudakov generalized this result, obtaining an asymptotic improvement when the sparsity parameter $f$ is large (for instance, when $f = \Theta(\Delta)$).\footnote{Throughout this work, we use standard asymptotic notation $O(\cdot)$, $\Omega(\cdot)$, $o(\cdot)$, etc.}

\begin{theorem}[{Alon, Krivelevich, and Sudakov~\cite{AKSConjecture}}]\label{theo:AKS}
    There exists a constant $C > 0$ such that for all sufficiently large $\Delta \in \N$ and all $f > 1$, if $G$ is a $\frac{\Delta^2}{f}$-locally-sparse graph of maximum degree $\Delta$, then $\chi(G) \le C \frac{\Delta}{\log f}$.
\end{theorem}

Furthermore, they showed that Theorem~\ref{theo:AKS} is optimal up to the implicit constant $C$~\cite[Proposition~1.2]{AKSConjecture}. Davies, Kang, Pirot, and Sereni later established that Theorem~\ref{theo:AKS} holds with $C = 2 + o(1)$, which remains the best-known bound~\cite[Theorem~5]{DKPS}; their result extends to the more general framework of DP-coloring (see \S\ref{subsection: DP def}).

In this paper, we study the linear programming relaxation of graph coloring. The \emph{fractional chromatic number}, $\chi_f(G)$, is the optimal value of the linear program:
\begin{mini}
    {\lambda \in \mathbb{R}^{|\cI(G)|}}{\sum_{I \in \cI(G)}\lambda_I}{\label{eq:chif}}{}
    \addConstraint{\sum_{\substack{I \in \cI(G) \\ v \in I}}\lambda_I}{\ge 1}{\qquad \forall v \in V(G)}
    \addConstraint{\lambda_I}{\ge 0}{\qquad \forall I \in \cI(G),}
\end{mini}
where $\cI(G)$ denotes the family of independent sets of $G$. Every proper coloring corresponds to an integral feasible solution to \eqref{eq:chif}, implying $\chi_f(G) \le \chi(G)$.

More generally, any feasible fractional coloring $\lambda$ of weight $q$ induces a probability distribution $\cD$ over $\cI(G)$ via $\pr_{I \sim \cD}[I = S] = \lambda_S / q$, satisfying $\pr_{I \sim \cD}[v \in I] \ge 1/q$ for all $v \in V(G)$. This probabilistic perspective is particularly effective, as upper bounds on $\chi_f(G)$ are frequently established by analyzing randomized procedures that sample independent sets (cf.~\cite{martinsson2025random, dhawan2026fractional, high_girth, davies2021occupancy, davies2017independent, davies2018average}).

Recently, Martinsson and Steiner~\cite{martinsson2025random} initiated the study of bounding $\chi_f(G)$ under degeneracy constraints, resolving a conjecture of Harris~\cite{harris2019some}:

\begin{theorem}[{Martinsson and Steiner~\cite{martinsson2025random}}]\label{theo: martinsson steiner}
    Let $G$ be a $d$-degenerate triangle-free graph. Then, $\chi_f(G) \le (4 + o(1))\frac{d}{\log d}$.
\end{theorem}

Crucially, there exist $d$-degenerate triangle-free graphs $G$ with $\chi(G) = d+1$~\cite{kostochka1999properties, descartes1954solution}, demonstrating that an analogous bound cannot hold for the ordinary chromatic number and revealing a sharp separation in the fractional setting.

The proof of Theorem~\ref{theo:AKS} relies on partitioning $V(G)$ into a small number of triangle-free induced subgraphs and applying a landmark result of Johansson on the chromatic number of triangle-free graphs to each part~\cite{Joh_triangle}.\footnote{Although widely cited as a DIMACS Technical Report, Johansson's original manuscript is no longer easily accessible. A complete treatment by Molloy and Reed~\cite{MolloyReed} establishes the explicit bound $\chi(G) \le 160 \frac{\Delta}{\log \Delta}$.} Martinsson and Steiner observed that a similar reduction allows Theorem~\ref{theo: martinsson steiner} to be extended to a broader notion of local sparsity:

\begin{definition}[Left Local Sparsity]\label{def: local sparsity}
    For $k \in \mathbb{R}$ and a graph $G = (V, E)$, a vertex ordering $(v_1, \dots, v_n)$ of $V(G)$ is \emph{left $k$-locally-sparse} if for each $i \in [n]$, the induced subgraph $G[N(v_i) \cap \{v_1, \dots, v_{i-1}\}]$ contains at most $\lfloor k \rfloor$ edges.
\end{definition}

For a $d$-degenerate graph admitting a left locally sparse degeneracy ordering, random sampling in the spirit of Alon, Krivelevich, and Sudakov~\cite{AKSConjecture} paired with Theorem~\ref{theo: martinsson steiner} yields the following general bound:\footnote{While not explicitly stated as a formal theorem in \cite{martinsson2025random}, Martinsson and Steiner outline this bound in \cite[p.~5]{martinsson2025random}.}

\begin{theorem}[{Martinsson and Steiner~\cite{martinsson2025random}}]\label{theo: martinsson steiner local}
    There exist constants $C, C' > 0$ such that for all sufficiently large $d \in \N$ and all $C\leq f \leq d^2 + 1$, if $G$ is a $d$-degenerate graph that admits a left $\frac{d^2}{f}$-locally-sparse degeneracy ordering, then $\chi_f(G) \le C' \frac{d}{\log f}$.
\end{theorem}

Recently, Davies~\cite{davies2026random} generalized the above bound to the setting of fractional coloring with local demands–wherein each vertex $v$ ``demands'' a certain probability of being included in the random independent set $I$–introduced by Kelly and Postle~\cite{kelly2024fractional}.
Optimizing the constant $C'$ in Theorem~\ref{theo: martinsson steiner local} remains an intriguing direction. Our first main result provides progress toward this goal.

\begin{theorem}\label{theorem: main ordinary coloring}
    There exists $C > 0$ such that for all sufficiently large $d$ and all $C \le f \le d^2 + 1$, if $G$ is a $d$-degenerate graph that admits a left $\frac{d^2}{f}$-locally-sparse degeneracy ordering, then
    \[
        \chi_f(G) \le (8 + o(1))\frac{d}{\log f}.
    \]
\end{theorem}

Notably, setting $f = d^2 + 1$ recovers Theorem~\ref{theo: martinsson steiner}. 
For $f \ge d^{\Omega(1)}$, this bound is tight up to the constant factor as there exist $d$-degenerate triangle-free graphs with $\chi_f(G) \ge (1 - o(1))\frac{d}{\log d}$~\cite{high_girth} (see Theorem~\ref{theo: lower bound degen girth}).
We show that this lower bound generalizes to small $f$ as well, implying that Theorem~\ref{theorem: main ordinary coloring} is sharp up to the leading constant.

\begin{theorem}\label{theo: lower bound}
    Let $\eps \in (0, 1)$ be arbitrary. There exist $C > 0$ and $d_0 \in \N$ such that for every $d \ge d_0$ and $f \ge C$, there exists a $d$-degenerate graph $G$ that admits a left $\frac{d^2}{f}$-locally-sparse degeneracy ordering such that
    \[
    \chi_f(G) \ge (1 - \eps)\frac{d}{\log f}.
    \]
\end{theorem}

A primary limitation of the random sampling framework introduced by Alon, Krivelevich, and Sudakov is its failure in list and correspondence (DP) coloring models---a limitation inherited by fractional variants. While Alon, Tuza, and Voigt~\cite{alon1997choosability} established that the fractional list chromatic number equals $\chi_f(G)$ for all graphs, no such collapse occurs for the fractional DP-chromatic number, which forms the central focus of this paper.

The remainder of this introduction is organized as follows: in \S\ref{subsection: DP def}, we define fractional DP-colorings and highlight key distinctions from ordinary fractional colorings; in \S\ref{subsection: main results}, we formally present our main result for locally sparse graphs and discuss applications to specific graph classes; in \S\ref{subsection: notation}, we collect basic notation used throughout the paper; finally, in \S\ref{subsection: overview}, we provide an informal overview of our proof techniques.

\subsection{Fractional DP-coloring}\label{subsection: DP def}

Correspondence coloring (also known as \emph{DP-coloring}) is a generalization of list coloring introduced by Dvo\v{r}\'{a}k and Postle~\cite{DPCol} to address a question of Borodin. In list coloring, each vertex $v \in V(G)$ is assigned a list of available colors $L(v)$, and a proper $L$-coloring is a proper coloring in which each vertex $v$ receives a color from $L(v)$. DP-coloring extends this framework by allowing color identifications to vary from edge to edge: for each edge $uv \in E(G)$, color pairs between $L(u)$ and $L(v)$ are linked by a matching $M_{uv}$ (which may be partial or empty). A proper correspondence coloring is a choice $\phi(v) \in L(v)$ for each $v \in V(G)$ such that $\phi(u)\phi(v) \notin M_{uv}$ for all $uv \in E(G)$. Formally, correspondence colorings are defined using an auxiliary graph called a correspondence cover.

\begin{definition}[Correspondence Cover]\label{def:corr_cov}
    A \emph{correspondence cover} (or \emph{DP-cover}) of a graph $G$ is a pair $\cH = (L, H)$, where $H$ is a graph and $L\,:\,V(G) \to 2^{V(H)}$ is a mapping satisfying:
    \begin{enumerate}[label=(\normalfont DP\arabic*), leftmargin=\leftmargin + 1\parindent]
        \item The collection $\{L(v)\,:\,v \in V(G)\}$ forms a partition of $V(H)$;
        \item\label{dp:list_independent} $L(v)$ is an independent set in $H$ for each $v \in V(G)$; and
        \item\label{dp:matching} For all $u, v \in V(G)$, the edge set of $H[L(u) \cup L(v)]$ is a matching, which is empty if $uv \notin E(G)$.
    \end{enumerate}
\end{definition}

We refer to the vertices of $H$ as \emph{colors}. For any $c \in V(H)$, let $L^{-1}(c)$ denote the \emph{underlying vertex} of $c$ in $G$, i.e., the unique vertex $v \in V(G)$ with $c \in L(v)$. Two colors $c, c' \in V(H)$ are said to \emph{correspond} to each other if $cc' \in E(H)$, denoted $c \sim c'$.

An \emph{$\cH$-coloring} of $G$ is a mapping $\phi\,:\,V(G) \to V(H)$ such that $\phi(v) \in L(v)$ for all $v \in V(G)$. Analogously, a \emph{partial $\cH$-coloring} is a partial function $\phi\,:\,V(G) \dashrightarrow V(H)$ satisfying $\phi(v) \in L(v)$ whenever $\phi(v)$ is defined. A (partial) $\cH$-coloring $\phi$ is \emph{proper} if its image is an independent set in $H$, meaning $\phi(u) \not\sim \phi(v)$ for all $u,v \in V(G)$ where both values are defined. Thus, the image of a proper $\cH$-coloring corresponds to an independent set $I \subseteq V(H)$ meeting each list $L(v)$ in exactly one vertex.

A DP-cover $\cH = (L,H)$ is \emph{$q$-fold} if $|L(v)| = q$ for all $v \in V(G)$. The \emph{DP-chromatic number} of $G$, denoted $\chi^{\mathrm{DP}}(G)$, is the minimum integer $q$ such that $G$ admits a proper $\cH$-coloring for every $q$-fold correspondence cover $\cH$.

A rich body of literature investigates the DP-chromatic number of graphs under local sparsity constraints~\cite{AndersonBernshteynDhawan, anderson2024coloring, anderson2025coloring, bonamy2022bounding, dhawan2024palette, dhawan2025bounds, DKPS}. Bonamy, Kelly, Nelson, and Postle~\cite{bonamy2022bounding} generalized Theorem~\ref{theo:AKS} to DP-coloring, building on work of Molloy~\cite{Molloy} and Bernshteyn~\cite{bernshteyn2019johansson} for $K_r$-free graphs. Davies, Kang, Pirot, and Sereni~\cite{DKPS} further refined these methods by developing the \emph{local occupancy framework}, which yields the current best-known asymptotic bounds (see also~\cite{davies2020algorithmic, davies2025hard, dhawan2025bounds}).

Bernshteyn, Kostochka, and Zhu~\cite{bernshteyn2020fractional} introduced the notion of fractional DP-coloring, extending fractional list coloring to the DP setting. We begin with the key concept of an $(\eta, \cH)$-coloring, in which a vertex is assigned an $\eta$-fraction of its list rather than a single color.

\begin{definition}\label{def: eta H col}
    Let $\cH = (L, H)$ be a DP-cover of a graph $G$ and let $\eta \in [0, 1]$. An \emph{$(\eta,\cH)$-coloring} of $G$ is an independent set $S \subseteq V(H)$ such that $|S \cap L(u)| \ge \eta |L(u)|$ for all $u \in V(G)$.
\end{definition}

We now define the fractional DP-chromatic number.

\begin{definition}\label{def: chiDP}
    Let $G$ be a graph. For each $q \in \N^+$, define
    \[
        \theta_{\mathrm{DP}}(G, q) \coloneqq \max\bigl\{\eta \in [0, 1]\,:\,G \text{ admits an } (\eta, \cH)\text{-coloring for every } q\text{-fold DP-cover } \cH\bigr\}.
    \]
    The \emph{fractional DP-chromatic number} of $G$ is
    \[
        \chiDP(G) \coloneqq \inf\bigl\{\theta_{\mathrm{DP}}(G, q)^{-1}\,:\,q \in \N^+\bigr\}.
    \]
\end{definition}

In their seminal paper introducing this notion, Bernshteyn, Kostochka, and Zhu~\cite{bernshteyn2020fractional} established several foundational results. They showed that bipartite graphs containing at most one cycle are precisely those satisfying $\chiDP(G) \le 2$. Furthermore, they proved a general lower bound in terms of the maximum average degree,
\[
    \mathrm{mad}(G) \coloneqq \max\left\{\frac{2|E(H)|}{|V(H)|}\,:\,H \subseteq G, \, V(H) \neq \emptyset\right\}:
\]

\begin{theorem}[{Bernshteyn, Kostochka, and Zhu~\cite{bernshteyn2020fractional}}]\label{theorem: BKZ MAD lb}
    If $G$ is a graph with maximum average degree $D \ge 4$, then $\chiDP(G) \ge \frac{D}{2\log D}$.
\end{theorem}

This implies that the gap between $\chi_f(G)$ and $\chiDP(G)$ can be arbitrarily large (as seen, for instance, by taking $G = K_{n, n}$). Additionally, they proved that this lower bound is tight up to a constant factor for bipartite graphs by providing a matching upper bound in terms of degeneracy:

\begin{theorem}[{Bernshteyn, Kostochka, and Zhu~\cite{bernshteyn2020fractional}}]\label{theorem: BKZ bipartite ub}
    If $G$ is a $d$-degenerate bipartite graph, then $\chiDP(G) \le (1+o(1))\frac{d}{\log d}$.
\end{theorem}

As $\mathrm{degen}(G) \le \mathrm{mad}(G) \le 2\mathrm{degen}(G) + 1$, Theorems~\ref{theorem: BKZ MAD lb} and~\ref{theorem: BKZ bipartite ub} imply a constant-factor gap for arbitrary bipartite graphs.
In fact, as $K_{d, d^2}$ is $d$-degenerate and has maximum average degree at least $(2-o(1))d$, the bound in Theorem~\ref{theorem: BKZ bipartite ub} is sharp; Bernshteyn, Kostochka, and Zhu construct an infinite family of bipartite graphs satisfying this (see the discussion following~\cite[Corollary~1.11]{bernshteyn2020fractional}).

\subsection{Main results for locally sparse graphs}\label{subsection: main results}

We are ready to present our main result and its applications to certain locally sparse graph classes. We begin with a DP-coloring analogue of Theorem~\ref{theorem: main ordinary coloring}.

\begin{theorem}\label{theorem: main theorem}
    Let $\eps \in (0, 1)$ be arbitrary.
    There exist $C > 0$ and $d_0 \in \N$ such that for every $d \ge d_0$ and $C \le f \le d^2 + 1$, if $G$ is a $d$-degenerate graph that admits a left $\frac{d^2}{f}$-locally-sparse degeneracy ordering, then
    \[
        \chiDP(G) \le (8 + \eps)\frac{d}{\log f}.
    \]
\end{theorem}

Since $\chi_f(G) \le \chiDP(G)$, Theorem~\ref{theorem: main theorem} immediately implies Theorem~\ref{theorem: main ordinary coloring}. Furthermore, setting $f = d^2 + 1$ extends Theorem~\ref{theo: martinsson steiner} to the DP-coloring setting. Additionally, in light of Theorem~\ref{theorem: BKZ MAD lb}, Theorem~\ref{theorem: main theorem} is tight up to a constant factor.

In the remainder of this section, we discuss several consequences of our main result. First, we consider $K_{1, t, t}$-free graphs, a well-studied graph class~\cite{AKSConjecture, anderson2024coloring, anderson2025coloring, DKPS, vu2002general}. If $G$ is $K_{1, t, t}$-free, the neighborhood of each vertex in $G$ does not contain $K_{t, t}$ as a subgraph. The celebrated \hyperref[theo:KST]{K\H{o}v\'ari--S\'os--Tur\'an Theorem} implies that such neighborhoods contain relatively few edges:

\begin{theorem}[{K\H{o}v\'ari--S\'os--Tur\'an Theorem~\cite{KovariSosTuran}}]\label{theo:KST}
    Let $G$ be a bipartite graph with parts $A$ and $B$ of sizes $m$ and $n$, respectively. If $G$ contains no subset $X \subseteq A$ of size $t$ and no subset $Y \subseteq B$ of size $s$ for which $G[X \cup Y] \cong K_{s,t}$, then
    \[
        e(G) < (t-1)^{1/s} n m^{1-1/s} + (s-1) m.
    \]
\end{theorem}

Treating $t$ as a constant with respect to $d$, Theorem~\ref{theo:KST} implies that every $d$-degenerate $K_{1, t, t}$-free graph $G$ is left $d^{2 - 1/(2t) + o(1)}$-locally sparse under any degeneracy ordering, yielding the following corollary to Theorem~\ref{theorem: main theorem}:

\begin{corollary}\label{corollary: K1tt}
    Let $t \in \N$ be arbitrary. There exists $d_0 \in \N$ such that for all $d \ge d_0$, if $G = (V, E)$ is a $d$-degenerate $K_{1, t, t}$-free graph, then
    \[
        \chiDP(G) \le (16t + o(1))\frac{d}{\log d}.
    \]
\end{corollary}

Alon, Krivelevich, and Sudakov~\cite{AKSConjecture} showed in 1999 that $K_{1, t, t}$-free graphs satisfy $\chi(G) = O\left(\frac{t\Delta}{\log \Delta}\right)$ as a corollary to Theorem~\ref{theo:AKS}, and famously conjectured that an analogous bound holds when $K_{1, t, t}$ is replaced by $K_t$ (with a potentially worse dependence on $t$). Recently, Methuku, Janzer, and the first author~\cite{Kttt} established that $\chi(G) = O\left(\frac{t^6\Delta}{\log \Delta}\right)$ for $K_{t, t, t}$-free graphs, marking the first progress toward this conjecture since it was posed. A key step in their argument shows that every $K_{t, t, t}$-free graph admits a left locally sparse vertex ordering:

\begin{lemma}[{Dhawan, Methuku, and Janzer~\cite{Kttt}}]\label{lemma: kttt}
    Let $G$ be a $K_{t, t, t}$-free graph with maximum degree $\Delta$. Then $G$ admits a left $\Delta^{2 - 1/(2t^2) + o(1)}$-locally-sparse vertex ordering.
\end{lemma}

Since every graph $G$ is $\Delta(G)$-degenerate, combining Theorem~\ref{theorem: main theorem} with Lemma~\ref{lemma: kttt} yields the following bound on the fractional DP-chromatic number:

\begin{corollary}\label{corollary: Kttt}
    Let $t \in \N$ be arbitrary. There exists $\Delta_0 \in \N$ such that for all $\Delta \ge \Delta_0$, if $G = (V, E)$ is a $K_{t, t, t}$-free graph with maximum degree $\Delta$, then
    \[
        \chiDP(G) \le (16t^2 + o(1))\frac{\Delta}{\log \Delta}.
    \]
\end{corollary}

Notably, this improves the bound on $\chi_f(G)$ implied by $\chi(G) = O\left(\frac{t^6\Delta}{\log \Delta}\right)$ by a factor of $\Theta(t^4)$.

\subsection{Notation}\label{subsection: notation}

Throughout the rest of the paper, we use the following standard notation.
We adopt the convention $0\log 0 = 0$, where $\log$ denotes the natural logarithm.
For a nonnegative integer $n$, we write $[n] \coloneqq \{1, \dots, n\}$, with $[0] \coloneqq \emptyset$.
We let $\bone{\mathcal{A}}$ denote the indicator function of an event $\mathcal{A}$.
For $p \in [0, 1]$, $\mathrm{Ber}(p)$ denotes the Bernoulli distribution on $\{0,1\}$ with parameter $p$.

For a graph $G$, its vertex and edge sets are denoted by $V(G)$ and $E(G)$, respectively.
For a vertex $v \in V(G)$, $N_G(v) \coloneqq \{u \in V(G)\,:\,uv \in E(G)\}$ denotes the neighborhood of $v$, and $\deg_G(v) \coloneqq |N_G(v)|$ denotes its degree; we drop the subscript $G$ when the underlying graph is clear from context.
For a subset $U \subseteq V(G)$, $G[U]$ denotes the subgraph of $G$ induced by $U$.

Given an ordering $(v_1, \ldots, v_n)$ of $V(G)$, we let $N_b(v_i)$ denote the set of neighbors $v_j$ of $v_i$ such that $j < i$ (the \emph{back-neighbors} of $v_i$).
Similarly, we let $N_f(v_i)$ denote the set of neighbors $v_k$ of $v_i$ such that $k > i$ (the \emph{front-neighbors} of $v_i$).
Given a DP-cover $\cH = (L, H)$ of $G$ and a vertex $c \in V(H)$, we let $N_b(c) \coloneqq \{c' \in N(c)\,:\,L^{-1}(c') \in N_b(L^{-1}(c))\}$; we define $N_f(c)$ similarly. Lastly, we let $N^{(j)}(c) \coloneqq N_H(c) \cap L(v_j)$ (note that $|N^{(j)}(c)| \leq 1$).

\subsection{Proof overview}\label{subsection: overview}

In this section, we provide an informal overview of our proof techniques.
We first briefly discuss the construction for Theorem~\ref{theo: lower bound}, and dedicate the remainder of the section to describing our main upper bound result, Theorem~\ref{theorem: main theorem}.

\subsubsection*{Lower Bound}

The proof of Theorem~\ref{theo: lower bound} is inspired by a construction of Alon, Krivelevich, and Sudakov~\cite{AKSConjecture}, which shows that Theorem~\ref{theo:AKS} is tight up to the leading constant.
Indeed, their construction yields a valid lower bound in our setting as well, though with a significantly weaker leading constant.
Crucially, our argument relies on a recent result of Allen, Noel, and the first author~\cite{high_girth}, who established the existence of $d$-degenerate triangle-free graphs with $\chi_f(G) \ge (1-o(1))\frac{d}{\log d}$ (Theorem~\ref{theo: lower bound degen girth}).
For $f \ge d^{1 - o(1)}$, this construction directly applies.
For smaller values of $f$, we let $H$ be an $f$-degenerate triangle-free graph with $\chi_f(H) \ge (1-o(1))\frac{f}{\log f}$ as guaranteed by Theorem~\ref{theo: lower bound degen girth}, and construct $G$ from $H$ by replacing each vertex of $H$ with a clique $K_t$ for $t \approx d/f$.
It is straightforward to derive a $d$-degenerate $\frac{d^2}{f}$-locally-sparse vertex ordering of $G$ from an $f$-degenerate ordering of $H$.
Furthermore, since $G = H \cdot K_t$ (where $\cdot$ denotes the lexicographic graph product), it follows that $\chi_f(G) = \chi_f(H)\chi_f(K_t) = t\chi_f(H)$, completing the proof.

\subsubsection*{Upper Bound}
We record the following observation in light of Definition~\ref{def: chiDP}:

\begin{obs}\label{obs: fractional coloring}
    Let $G$ be a graph and let $\eta \in [0, 1]$. If there exists $q \in \N$ such that $G$ admits an $(\eta, \cH)$-coloring for every $q$-fold DP-cover $\cH$, then $\chiDP(G) \le \eta^{-1}$.
\end{obs}

In our proof, we take advantage of this observation.
In particular, we will show that the conditions of Observation~\ref{obs: fractional coloring} are satisfied for
\begin{equation}\label{eq: overview eta q0}
    \eta = (1 - o(1))\frac{\log f}{8d} \qquad \text{and} \qquad q_0 = n^{3n}.
\end{equation}
To this end, we design an algorithm that takes as input a $d$-degenerate graph $G$, a left $\frac{d^2}{f}$-locally-sparse degeneracy ordering $(v_1, \dots, v_n)$ of $V(G)$, and a DP-cover $\cH = (L, H)$ of $G$, and outputs an independent set $S \subseteq V(H)$.
We will argue that, with high probability, $S$ is an $(\eta, \cH)$-coloring provided that $\cH$ is $q$-fold for $q \ge q_0$.

Let us now describe our procedure.
The algorithm is inspired by recent work of the first author~\cite{dhawan2026fractional}, which provided an alternate proof and several extensions of Martinsson and Steiner's result (Theorem~\ref{theo: martinsson steiner}).\footnote{We note that the argument of Martinsson and Steiner can be adapted to prove a fractional DP-coloring version of Theorem~\ref{theo: martinsson steiner}. However, their approach fails for locally sparse graphs due to the nonlinear nature of their activation probabilities; see~\cite[pp.~1--2]{dhawan2026fractional} for a detailed discussion.}
We tailor the algorithm to the DP-coloring setting:
\begin{enumerate}[label=(\arabic*)]
    \item Assign each color $c \in V(H)$ an initial weight $p_0(c) = \alpha \approx 2\eta$. At step $i$, for each $c \in L(v_i)$, include $c \in S$ with probability $p_{i-1}(c)$.
    
    \item If $c \in S$, set $p_i(c') = 0$ for all front-neighbors $c' \in N_f(c)$ to maintain independence. If $c \notin S$, increase $p_i(c')$ to compensate and boost the future selection probability of $c'$.
    
    \item To ensure $p_i(c') \le 1$, weights exceeding a threshold $\hp$ are marked as \textit{bad} and excluded from $S$. Our updates preserve the conditional expectation $\E[p_i(c') \mid p_{i-1}(\cdot)] = p_{i-1}(c')$.
    Letting $B_i(v_j) \subseteq L(v_j)$ denote the set of bad colors in $L(v_j)$ after the $i$-th iteration, we conclude
    \[
    \E[|S \cap L(v_i)|] \gtrsim \alpha|L(v_i)| - \hp\E[|B_{i-1}(v_i)|].
    \]
\end{enumerate}

It now suffices to argue that (i) $\E[|B_{i-1}(v_i)|] \lesssim \alpha|L(v_i)|/2$, and (ii) the random variable $|S \cap L(v_i)|$ is concentrated.
We employ an entropy-based argument for the former goal.
Given a probability distribution $\cD$ over some finite set $T$, the entropy of $\cD$ is 
\[
H(\cD) = -\sum_{t\in T}\pr_{X \sim\cD}[X = t]\log \pr_{X \sim\cD}[X = t].
\]
Roughly speaking, $H(\cD)$ measures how widely the assignment probabilities vary---the more they vary, the smaller their entropy; $H(\cD)$ is maximized when $\cD$ is the uniform distribution.
We define the \textit{entropy} at $v$ after processing $i$ vertices as
\[
H_i(v) = -\sum_{c \in L(v)}p_i(c)\log p_i(c).
\]
If $\E[H_{i-1}(v_i)]$ is sufficiently large, the values of $p_{i-1}(c)$ do not vary drastically in expectation.
Since $\sum_{c \in L(v)}\E[p_i(c)] = \alpha|L(v)|$, it follows that few of these values reach $\hp$ in expectation.

In our proof, we show that the entropy does not decrease significantly in expectation:
\[
H_0(v_i) - \E[H_{i-1}(v_i)] \lesssim \alpha|L(v_i)|\left(\alpha|N_b(v_i)| + |E(G[N_b(v_i)])|\hp^2\right).
\]
For triangle-free graphs, the right-hand side reduces to $\alpha^2 d|L(v_i)|$ because $E(G[N_b(v_i)]) = \emptyset$ for all $v_i$.
While this no longer holds when triangles are permitted, we leverage $\frac{d^2}{f}$-local-sparsity. By setting $\hp = \frac{f^{1/2 - o(1)}}{d}$, the first term dominates, effectively mitigating the contribution of triangles.

Once we establish that $\E[|S\cap L(v_i)|]$ is sufficiently large, we bound $|S\cap L(v_i)| \gtrsim \eta|L(v_i)|$ with high probability using standard concentration tools.
In particular, we rely on the following classical consequence of Azuma's inequality for Doob martingales:

\begin{theorem}[{\cite[p.~79]{MolloyReed}}]\label{theo: tal}
    Let $X$ be a random variable determined by $s$ independent trials such that changing the outcome of any single trial affects $X$ by at most $\zeta$. Then,
    \[
    \pr[|X - \E[X]| \ge t] \le \exp\left(-\frac{t^2}{2\zeta^2s}\right).
    \]
\end{theorem}

Concentration arguments are straightforward in ordinary coloring, but DP-coloring introduces dependencies that require additional care. 
Specifically, the events $\{c \in S\}$ and $\{c' \in S\}$ for distinct $c, c' \in L(v)$ are not necessarily independent for arbitrary cover graphs.
Nevertheless, we show that the random variable $|S \cap L(v)|$ satisfies the conditions of Theorem~\ref{theo: tal} with parameters $s = nq$ and $\zeta = n^n$, yielding the desired concentration whenever $q \ge q_0$.

\section{Lower Bound: Proof of Theorem~\ref{theo: lower bound}}\label{section: lower bounds}

In this section, we prove our main lower bound result, Theorem~\ref{theo: lower bound}. For convenience, we restate the result below.

\begin{theorem*}[Restatement of Theorem~\ref{theo: lower bound}]
    Let $\eps \in (0, 1)$ be arbitrary. There exist $C > 0$ and $d_0 \in \N$ such that for every $d \ge d_0$ and $f \ge C$, there exists a $d$-degenerate graph $G$ that admits a left $\frac{d^2}{f}$-locally-sparse degeneracy ordering such that
    \[
    \chi_f(G) \ge (1 - \eps)\frac{d}{\log f}.
    \]
\end{theorem*}

We rely on the following recent result by Allen, Noel, and the first author:

\begin{theorem}[{Allen--Dhawan--Noel~\cite{high_girth}}]\label{theo: lower bound degen girth}
    Let $\gamma \in (0, 1)$. There exists $d_0 \in \mathbb{N}$ such that for all $d \ge d_0$, there exists a $d$-degenerate, $K_3$-free graph $G$ satisfying
    \[
    \chi_f(G) \ge (1-\gamma)\frac{d}{\log d}.
    \]
\end{theorem}

Recall that the lexicographic product (or graph composition) $G \cdot H$ of two graphs $G$ and $H$ has vertex set $V(G) \times V(H)$, where two vertices $(u, v)$ and $(x, y)$ are adjacent if and only if $u x \in E(G)$, or $u = x$ and $v y \in E(H)$. The following standard property is well known (see, e.g.,~\cite{klavẑar1998fractional}):

\begin{fact}\label{fact: blow up}
    For any graphs $G$ and $H$, we have $\chi_f(G \cdot H) = \chi_f(G) \cdot \chi_f(H)$.
\end{fact}

With these preliminary results in place, we proceed to the proof of Theorem~\ref{theo: lower bound}.

\begin{proof}[Proof of Theorem~\ref{theo: lower bound}]
    If $f \ge d^{1 - \varepsilon/10}$, the claim follows immediately by applying Theorem~\ref{theo: lower bound degen girth} with parameter $\gamma = \varepsilon/2$, since $K_3$-free graphs serve as valid constructions in this regime. Thus, we may assume $f < d^{1 - \varepsilon/10}$.
    
    Let $H$ be a $\ceil{2f}$-degenerate, $K_3$-free graph with fractional chromatic number satisfying
    \[
    \chi_f(H) \ge \left(1-\frac{\varepsilon}{2}\right)\frac{\ceil{2f}}{\log \ceil{2f}},
    \]
    as guaranteed by Theorem~\ref{theo: lower bound degen girth} for $f \ge C$.

    Setting $t \coloneqq \floor{\frac{d}{\ceil{2f} + 1}} \ge d^{\varepsilon/20}$, define $G \coloneqq H \cdot K_t$. We show that $G$ fulfills all required conditions. By Fact~\ref{fact: blow up},
    \[
    \chi_f(G) = \chi_f(H) \cdot \chi_f(K_t) \ge \left(1-\frac{\varepsilon}{2}\right)\frac{\ceil{2f}}{\log \ceil{2f}} \cdot t \ge (1 - \varepsilon)\frac{d}{\log f}
    \]
    for all sufficiently large $d,\,f$.

    It remains to show that $G$ admits a $d$-degenerate left $\frac{d^2}{f}$-locally-sparse ordering. Let $n = |V(H)|$ and let $(v_1, \ldots, v_n)$ be an $f$-degenerate ordering of $V(H)$. Identifying $V(K_t)$ with $[t]$, we order $V(G)$ lexicographically:
    \[
    ((v_1, 1), \ldots, (v_1, t), (v_2, 1), \ldots, (v_n, t)).
    \]
    
    For a vertex $u = (v_i, j)$, its preceding neighbors in this ordering are either of the form $(v_i, j')$ for $j' < j$, or $(v_l, j')$ for $l < i$ where $v_l v_i \in E(H)$. Consequently, the maximum back-degree of any vertex is bounded by
    \[
    (t - 1) + \ceil{2f}t \le t(\ceil{2f} + 1) \le d,
    \]
    confirming that the ordering is $d$-degenerate.

    Furthermore, the back-neighborhood of $u$ consists of at most $t-1$ vertices in $\{v_i\} \times [t]$ and $t$ vertices in $\{v_l\} \times [t]$ for each back-neighbor $v_l$ of $v_i$. Because $H$ is $K_3$-free, no edges exist between distinct back-neighbors of $v_i$ in $H$. Therefore, the total number of edges induced by the back-neighborhood of $u$ is at most
    \[
    \binom{t-1}{2} + \ceil{2f}\binom{t}{2} + \ceil{2f}(t-1)t \le \frac{t^2(\ceil{2f} + 1)}{2} + \ceil{2f}t^2 \le \frac{d^2}{f},
    \]
    which completes the proof.
\end{proof}

\section{Upper bound: proof of Theorem~\ref{theorem: main theorem}}\label{section: main proof}

In this section, we prove our main upper bound result, Theorem~\ref{theorem: main theorem}.
For the reader's convenience, we restate the result below.

\begin{theorem*}[Restatement of Theorem~\ref{theorem: main theorem}]
    Let $\eps \in (0, 1)$ be arbitrary.
    There exist $C > 0$ and $d_0 \in \N$ such that for every $d \ge d_0$ and $C \le f \le d^2 + 1$, if $G$ is a $d$-degenerate graph that admits a left $\frac{d^2}{f}$-locally-sparse degeneracy ordering, then
    \[
        \chiDP(G) \le (8 + \eps)\frac{d}{\log f}.
    \]
\end{theorem*}

Throughout the proof, we tacitly assume that $f$ is sufficiently large with respect to $\eps$. We introduce the following parameters:
\begin{equation}\label{eq: eta hp}
    \eta \coloneqq \frac{1}{(8 + \eps)}\frac{\log f}{d}\qquad\text{and}\qquad\hp \coloneqq \frac{f^{\frac{1}{2} - \frac{\eps}{40}}}{d}.
\end{equation}
Recall from \S\ref{subsection: overview} that our goal is to construct an $(\eta,\,\cH)$-coloring for every $q$-fold DP-cover $\cH$ of $G$ for all sufficiently large $q$.
We begin with a description of our algorithm.

\medskip

\begin{breakablealgorithm}
\caption{Random Independent Sets in DP-Covers}
\label{alg:dp-entropy}

\begin{flushleft}
\textbf{Input}: 
An \(n\)-vertex graph \(G=(V,E)\), a \(q\)-fold DP-cover \(\cH=(L,H)\) for some \(q\in\N\), a parameter \(\eps>0\), and a $d$-degenerate left $\frac{d^2}{f}$-locally-sparse vertex ordering $(v_1, \ldots, v_n)$.\\
\textbf{Output}: 
An independent set $S \subseteq V(H)$.
\end{flushleft}
\begin{itemize}
    \item \textbf{Initialize:}
    For each \(v\in V(G)\), set \(B_0(v)=\emptyset\) and \(S=\emptyset\). For each \(c\in V(H)\), set \(p_0(c)=\alpha \defeq \frac{2\eta}{1 - \eps/20}\), where \(\eta\) is defined in~\eqref{eq: eta hp}.
    \item \textbf{Iterate:}
    For \(i=1,\dots,n\), do the following:
    \begin{enumerate}
        \item\label{step: not neighbor} For each \(c'\in V(H)\) such that \(N^{(i)}(c') \setminus B_{i-1}(v_i) = \emptyset\), set \(p_i(c')=p_{i-1}(c')\).
        \item For each \(c\in L(v_i)\setminus B_{i-1}(v_i),\) do the following: 
        \begin{enumerate}[ref=\theenumi\theenumii]
            \item\label{step: sample} Let \(a_c\sim \Ber(p_{i-1}(c))\). If $a_c = 1$, then add $c$ to $S$.
            \item\label{step: update} For each $c'\in N_b(c)$, let $p_i(c') = p_i(c)$, and for each \(c'\in N_f(c)\), make the following update:
            \[p_i(c') = \left\{\begin{array}{cc}
               p_{i-1}(c')  & \text{if } p_{i-1}(c') > \hp; \\[1em]
               0  & \text{if } a_c = 1; \\[1em]
               \dfrac{p_{i-1}(c')}{1-p_{i-1}(c)}  & \text{otherwise.}
            \end{array}\right.\]
        \end{enumerate}
        \item For each \(j>i\), set \(B_i(v_j)=\{c'\in L(v_j)\,:\, p_i(c')> \hp \}\). 
    \end{enumerate}
    
\end{itemize}
\end{breakablealgorithm}

\medskip

A few remarks are in order.
First, as a result of our update rule (step~\ref{step: update}), the output $S$ is an independent set.
Furthermore, once the weight of a color exceeds $\hp$, it remains unchanged thereafter.
It suffices to show that there is an outcome of Algorithm~\ref{alg:dp-entropy} satisfying $|S\cap L(v)| \geq \eta|L(v)|$ for each $v \in V(G)$.
To this end, let $S(v) \coloneqq S \cap L(v)$ be the random set of colors assigned to $v$.
We will establish that these sets are large via two key lemmas.

\begin{lemma}\label{lemma: main1}  
    For each $k \in [n]$, \(\E[|S(v_k)|]\geq \alpha q /2\).
\end{lemma}

\begin{lemma}\label{lemma: main2}
    For each $k \in [n]$, \(\pr[|S(v_k)|\leq (1- \eps/20)\alpha q/2]=\exp\left( -\Omega\left( \dfrac{\varepsilon^2q\alpha^2}{n^{2n + 1}}\right)  \right)\).
\end{lemma}

We defer the proofs of these lemmas to \S\ref{subsection: exp S} and \S\ref{subsection: conc S}, respectively.
With these results in hand, we observe that
\[\pr\left[\exists k \text{ s.t. }|S(v_k)| \leq \eta|L(v_k)|\right] = \pr\left[\exists k \text{ s.t. }|S(v_k)| \leq (1- \eps/20)\alpha q/2\right] \leq n\exp\left(-\Omega\left(\frac{\varepsilon^2q\alpha^2}{n^{2n + 1}}\right)\right).\]
This probability vanishes as $q \to \infty$, which completes the proof of Theorem~\ref{theorem: main theorem} by Observation~\ref{obs: fractional coloring}.

\subsection{Proof of Lemma~\ref{lemma: main1}}\label{subsection: exp S}
Note the following:

\begin{equation}
    \E[|S(v_k)]|=\E\left[\sum_{c\in L(v_k)\setminus B_{k-1}(v_k)}p_{k-1}(c)\right]=\E\left[\sum_{c\in L(v_k)}p_{k-1}(c)\right]-\E\left[\sum_{c \in B_{k-1}(v_k)}p_{k-1}(c)\right]\label{eq: main}.
\end{equation}
In the remainder of this section, we bound each term on the right-hand side. To this end,
we define the following quantities:
\begin{align*}
    P_i(v_k)&=\sum_{c\in L(v_k)}p_i(c), &\text{for }&0\leq i<k\leq n;\\
    Q_i(v_j,v_k)&=\sum_{c\in L(v_k)}\sum_{c'\in N^{(j)}(c)}\bone{c'\notin B_i(v_j)} p_i(c)p_i(c'),&\text{for }&0\leq i<j<k\leq n;\\
    H_i(v_k)&=-\sum_{c\in L(v_k)}p_i(c)\log p_i(c),&\text{for }&0\leq i<k\leq n.
\end{align*}
We refer to $P_i(v_k)$ as the \textit{potential} at $v_k$, $Q_i(v_j, v_k)$ as the \textit{energy} of the pair $(v_j, v_k)$, and $H_i(v_k)$ as the \textit{entropy} of $v_k$ after the $i$-th iteration.
We begin by computing the expected values of these random variables, starting with the potential.

\begin{lemma}\label{lemma: P}
    \(\E[P_{k-1}(v_k)]=\alpha q\).
\end{lemma}

\begin{proof}\stepcounter{ForClaims}\renewcommand{\theForClaims}{\ref*{lemma: P}}
    Let $\cF_i$ denote the $\sigma$-algebra generated by the history of the procedure up to the end of the $i$-th iteration, with $\cF_0$ representing the trivial $\sigma$-algebra. We use this notation throughout the proofs in this section. We claim that $(p_i(c))_{i=0}^{k-1}$ is a martingale with respect to $(\cF_i)_{i=0}^{k-1}$ for each $c \in L(v_k)$.
    \begin{claim}\label{claim: weight expectation}
        For every \(1\leq i<k\), \(\E[p_i(c)\mid \cF_{i-1}]=p_{i-1}(c).\)
    \end{claim}
    \begin{claimproof}
    As a result of step~\ref{step: not neighbor}, 
    if $N^{(i)}(c) \setminus B_{i-1}(v_i) = \emptyset$, the claim is trivial.
    Suppose $c' \in N^{(i)}(c) \setminus B_{i-1}(v_i)$.
    If $c \in B_{i-1}(v_j)$, the claim follows from step~\ref{step: update}.
    Otherwise, we have
    \[
    \E[p_i(c)\mid \cF_{i-1}]=(1-p_{i-1}(c'))\frac{p_{i-1}(c)}{1-p_{i-1}(c')}=p_{i-1}(c),
    \]
    as desired.
    \end{claimproof}
    
    Repeatedly applying Claim~\ref{claim: weight expectation} yields \(\E[p_{k-1}(c)]=\E[p_0(c)]=\alpha\), which gives \(\E[P_{k-1}(v_k)]=\alpha q\). 
\end{proof}

Next, we bound the energy of a pair of vertices.

\begin{lemma}\label{lemma: Q}
    For $1 \leq j < k \leq n$, we have
    \[\E[Q_{j-1}(v_j,v_k)]\leq  \sum_{c\in L(v_k)}\sum_{c' \in N^{(j)}(c)}\left(\alpha^2 + 2\hp^2\alpha|N_b(c)\cap N_b(c')|\right).\]
\end{lemma}

\begin{proof}\stepcounter{ForClaims}\renewcommand{\theForClaims}{\ref*{lemma: Q}}
By the definition of energy, for \(1 \leq j < k \leq n\) we have:
\begin{equation}
    \E[Q_{j-1}(v_j,v_k)]=\sum_{c\in L(v_k)}\sum_{c'\in N^{(j)}(c)} \E[\bone{c' \notin B_{j-1}(v_j)}p_{j-1}(c)p_{j-1}(c')].\label{eq: Q}
\end{equation}
It suffices to bound \(\E[\bone{c' \notin B_{j-1}(v_j)}p_{j-1}(c)p_{j-1}(c')]\) from above for each \(c\in L(v_k)\) and \(c'\in L(v_j)\) such that \(cc'\in E(H)\).
To this end, we define the following random variable:
\[S_i \coloneqq \bone{c' \notin B_{i}(v_j)}\left(p_{i}(c)p_{i}(c') + 2\hp^2p_{i}(c)|\cup_{l=i + 1}^{j-1}(N^{(l)}(c) \cap N^{(l)}(c'))|\right).\]
Observe that $S_{j-1}=\bone{c' \notin B_{j-1}(v_j)}p_{j-1}(c)p_{j-1}(c')$, since the union above is empty. The critical claim is that this sequence of random variables forms a supermartingale.

\begin{claim}\label{claim: energy}
    For $i = 1, \ldots, j-1$, we have $\E[S_i \mid \cF_{i-1}] \leq S_{i-1}$.
\end{claim}

\begin{claimproof}
    Given $i\in[j-1]$, we have
    \begin{equation}\E[S_i \mid \cF_{i-1}] =\E\left[\bone{c' \notin B_{i}(v_j)}\left(p_{i}(c)p_{i}(c') + 2\hp^2p_{i}(c)|\cup_{l=i + 1}^{j-1}(N^{(l)}(c) \cap N^{(l)}(c'))|\right) \,\,\bigg|\,\, \cF_{i-1}\right]. \label{eq: exp bound}
    \end{equation}
    By Claim~\ref{claim: weight expectation} and since $\bone{c' \notin B_{i}(v_j)} \leq \bone{c' \notin B_{i-1}(v_j)}$ as a result of step~\ref{step: update}, we obtain
    \begin{align*}
        &~\E[\bone{c' \notin B_{i}(v_j)}2\hp^2p_{i}(c)|\cup_{l=i + 1}^{j-1}(N^{(l)}(c) \cap N^{(l)}(c'))| \mid \cF_{i-1}] \\
        &\qquad \qquad \leq \bone{c' \notin B_{i-1}(v_j)}2\hp^2|\cup_{l=i + 1}^{j-1}(N^{(l)}(c) \cap N^{(l)}(c'))| \E[p_{i}(c)\mid \cF_{i-1}] \\
        &\qquad \qquad = \bone{c' \notin B_{i-1}(v_j)}2\hp^2p_{i-1}(c)|\cup_{l=i + 1}^{j-1}(N^{(l)}(c) \cap N^{(l)}(c'))|.
    \end{align*}
    It now suffices to show that
    \[\E[\bone{c' \notin B_{i}(v_j)}p_{i}(c)p_{i}(c')\mid \cF_{i-1}] \leq \bone{c' \notin B_{i-1}(v_j)}\left(p_{i-1}(c)p_{i-1}(c') + 2\hp^2p_{i-1}(c)|N^{(i)}(c) \cap N^{(i)}(c')|\right).\]
    
First, note that if $c' \in B_{i-1}(v_j)$, both sides equal $0$ and the claim holds trivially.
Additionally, if $c \in B_{i-1}(v_k)$, applying Claim~\ref{claim: weight expectation} gives
\begin{align*}
    \E[p_i(c)p_i(c')\mid \cF_{i-1}] &= p_{i-1}(c)\E[p_i(c')\mid \cF_{i-1}] = p_{i-1}(c)p_{i-1}(c') \\
    &\leq p_{i-1}(c)p_{i-1}(c') + 2\hp^2p_{i-1}(c)|N^{(i)}(c) \cap N^{(i)}(c')|,
\end{align*}
as desired.
A similar argument applies if $N^{(i)}(c) \setminus B_{i-1}(v_i) = \emptyset$ or $N^{(i)}(c') \setminus B_{i-1}(v_i) = \emptyset$.
Therefore, we may assume that $c \notin B_{i-1}(v_k)$, $c' \notin B_{i-1}(v_j)$, $N^{(i)}(c)  \setminus B_{i-1}(v_i) \neq \emptyset$, and $N^{(i)}(c')  \setminus B_{i-1}(v_i) \neq \emptyset$.
Let $c_1$ be the unique color in $N^{(i)}(c) \setminus B_{i-1}(v_i)$ and let $c_2$ be the unique color in $N^{(i)}(c') \setminus B_{i-1}(v_i)$.
We consider two cases:
\begin{itemize}
    \item \textbf{Case 1:} \(c_1 \neq c_2.\) 
    An argument identical to the proof of Claim~\ref{claim: weight expectation} yields
    \begin{align*}
        \E[p_i(c)p_i(c')\mid \cF_{i-1}] &= (1 - p_{i-1}(c_1))(1 - p_{i-1}(c_2))\frac{p_{i-1}(c)}{(1 - p_{i-1}(c_1))}\frac{p_{i-1}(c')}{(1 - p_{i-1}(c_2))} \\
        &= p_{i-1}(c) p_{i-1}(c'),
    \end{align*}
    as desired.
    \item \textbf{Case 2:} \(c_1 = c_2 = c''\).
    We have
    \begin{align*}
        \E[p_i(c)p_i(c') \mid \cF_{i-1}]
        &= \frac{p_{i-1}(c)}{(1- p_i(c''))} \frac{p_{i-1}(c')}{(1- p_i(c''))} (1- p_i(c'')) \\
        &=\frac{p_{i-1}(c)p_{i-1}(c')}{1-p_i(c'')}\\
        &\le \frac{p_{i-1}(c) p_{i-1}(c')}{1- \hp },
    \end{align*}
    where we used $p_{i-1}(c'') \leq \hp$ since $c'' \notin B_{i-1}(v_i)$.
    Because $c' \notin B_{i-1}(v_j)$ and $\hp < 1/2$, the above quantity is at most
    \[p_{i-1}(c) p_{i-1}(c')(1 + 2\hp) \leq p_{i-1}(c) p_{i-1}(c') + 2\hp^2p_{i-1}(c),\]
    as desired.
\end{itemize}
This covers all cases and completes the proof of the claim.
\end{claimproof}

By repeatedly applying Claim~\ref{claim: energy}, we obtain
\[
\E[\bone{c' \notin B_{j-1}(v_j)}p_{j-1}(c)p_{j-1}(c')]\leq S_0 = p_0(c)p_0(c') + 2\hp^2p_{0}(c)|N_b(c)\cap N_b(c')|.
\]
Substituting this bound into~\eqref{eq: Q}, we conclude that
\begin{align}
    \E[Q_{j-1}(v_j,v_k)]&=\sum_{c\in L(v_k)}\sum_{c'\in N^{(j)}(c)} \E[\bone{c' \notin B_{j-1}(v_j)}p_{j-1}(c)p_{j-1}(c')]\\
    &\leq\sum_{c\in L(v_k)}\sum_{c'\in N^{(j)}(c)} \left(p_0(c)p_0(c')+ 2\hp^2p_{0}(c)|N_b(c)\cap N_b(c')|\right)\\
    &=\sum_{c\in L(v_k)}\sum_{c'\in N^{(j)}(c)} \left(\alpha^2 + 2\hp^2\alpha|N_b(c)\cap N_b(c')|\right),
\end{align}
as desired.
\end{proof}

Finally, we bound the expected entropy.

\begin{lemma}\label{lemma: H}
    \(\E[H_{k-1}(v_k)]\geq H_0(v_k)-\frac{1}{1-\hp}\alpha q\left(\alpha d + \frac{2d^2\hp^2}{f}\right)\)
\end{lemma}

\begin{proof}\stepcounter{ForClaims}\renewcommand{\theForClaims}{\ref*{lemma: H}}
Recall that 
\begin{equation}
\E[H_{k-1}(v_k)]=\E\left[-\sum_{c\in L(v_k)}p_{k-1}(c)\log p_{k-1}(c)\right] =-\sum_{c\in L(v_k)}\E[p_{k-1}(c)\log p_{k-1}(c)]. \label{eq: H}
\end{equation}
The following claim is central to the argument:
\begin{claim}\label{claim: entropy}
    Fix \(c\in L(v_k)\). For \(1 \leq i < k\), we have 
    \[\E[p_i(c)\log p_i(c)\mid \cF_{i-1}]\leq p_{i-1}(c)\log p_{i-1}(c)+\sum_{c' \in N^{(i)}(c)}\bone{c' \notin B_{i-1}(v_i)}\frac{p_{i-1}(c)p_{i-1}(c')}{1-\hp }.\]
\end{claim}
\begin{claimproof}
The claim holds trivially if \(N^{(i)}(c) \setminus B_{i-1}(v_i)=\emptyset\), so we may assume there exists a unique \(c'\in N^{(i)}(c)\setminus B_{i-1}(v_i)\). Defining \(\varphi(x)=x\log x\) with the convention \(\varphi(0)=0\log 0 =0\), we have
\begin{align*}
    \E[\varphi(p_i(c))\mid \cF_{i-1}]&= p_{i-1}(c')\varphi(0) + (1-p_{i-1}(c'))\varphi\left(\frac{p_{i-1}(c)}{1-p_{i-1}(c')}\right)\\
    &= p_{i-1}(c)\log\left(\frac{p_{i-1}(c)}{1-p_{i-1}(c')}\right)\\
    &=\varphi(p_{i-1}(c))+ p_{i-1}(c)\log\left(\frac{1}{1-p_{i-1}(c')}\right).
\end{align*}
Since \(c' \notin B_{i-1}\) (meaning \(p_{i-1}(c')\leq\hp \)), applying the inequality \(-\log(1-x)\leq \frac{x}{1-x}\) for \(0\leq x<1\) yields 
\[
p_{i-1}(c)\log\left(\frac{1}{1-p_{i-1}(c')}\right)\leq \frac{p_{i-1}(c)p_{i-1}(c')}{1-p_{i-1}(c')}\leq 
\frac{p_{i-1}(c)p_{i-1}(c')}{1-\hp },\]
which completes the proof of the claim.
\end{claimproof}

Repeatedly applying Claim~\ref{claim: entropy}, we obtain
\begin{align*}
  \E[p_{k-1}(c)\log p_{k-1}(c)]&\leq p_0(c)\log p_0(c) + \frac{1}{1 - \hp}\sum_{j=1}^{k-1} \sum_{c' \in N^{(j)}(c)}\E[\bone{c' \notin B_{j-1}(v_j)}p_{j-1}(c)p_{j-1}(c')]. 
\end{align*}
Substituting this into~\eqref{eq: H} gives
\begin{align*}
    \E[H_{k-1}(v_k)]&=-\sum_{c\in L(v_k)}\E[p_{k-1}(c)\log p_{k-1}(c)] \\
    &\geq -\sum_{c\in L(v_k)}\left(p_0(c)\log p_0(c) + \frac{1}{1 - \hp}\sum_{j=1}^{k-1} \sum_{c' \in N^{(j)}(c)}\E[\bone{c' \notin B_{j-1}(v_j)}p_{j-1}(c)p_{j-1}(c')]\right)\\
    &=H_0(v_k)-\frac{1}{1-\hp }\sum_{c\in L(v_k)}\sum_{j=1}^{k-1} \sum_{c' \in N^{(j)}(c)}\E[\bone{c' \notin B_{j-1}(v_j)}p_{j-1}(c)p_{j-1}(c')]\\
    &= H_0(v_k)-\frac{1}{1-\hp }\sum_{j=1}^{k-1}\E[Q_{j-1}(v_j, v_k)].
\end{align*}
Applying Lemma~\ref{lemma: Q}, we have 
\begin{align*}
\E[H_{k-1}(v_k)]&\geq H_0(v_k)-\frac{1}{1-\hp }\sum_{j=1}^{k-1}\sum_{c\in L(v_k)}\sum_{c' \in N^{(j)}(c)}\left(\alpha^2 + 2\hp^2\alpha|N_b(c)\cap N_b(c')|\right) \\
&= H_0(v_k)-\frac{1}{1-\hp }\left(\sum_{c\in L(v_k)}\left(\alpha^2|N_b(c)| + 2\hp^2\alpha|E(H[N_b(c)])|\right)\right).
\end{align*}
Notice that $|N_b(c)| \leq |N_b(v_k)| \leq d$.
Furthermore, because $G$ is left \(\frac{d^2}{f}\)-locally-sparse, standard arguments show that $|E(H[N_b(c)])| \leq \frac{d^2}{f}$ (see, e.g., \cite[Proposition~11]{anderson2024coloring}).
It follows that
\[\E[H_{k-1}(v_k)] \geq H_0(v_k)-\frac{1}{1-\hp}\alpha q\left(\alpha d + \frac{2d^2\hp^2}{f}\right),\]
as claimed.
\end{proof}

With the above lemmas in hand, let us now show that the contribution of bad vertices to~\eqref{eq: main} is small.

\begin{lemma}\label{lemma : B}
    \( \E\left[\sum_{c\in B_{k-1}(v_k)}p_{k-1}(c)\right] \le \alpha q/2\).
\end{lemma}

\begin{proof}\stepcounter{ForClaims}\renewcommand{\theForClaims}{\ref*{lemma : B}}
First, we have the following simplification: 
\begin{align*}
    \E\left[\sum_{c\in B_{k-1}(v_k)}p_{k-1}(c)\right] &= \sum_{c\in L(v_k)}\E[p_{k-1}(c)\bone{c\in B_{k-1}(v_k)}]\\
    &=\sum_{c\in L(v_k)}\frac{\E[p_{k-1}(c)\log(\hp /\alpha)\bone{p_{k-1}(c)>\hp }]}{\log(\hp /\alpha)}\\
    &\leq\sum_{c\in L(v_k)}\frac{\E[p_{k-1}(c)\log (p_{k-1}(c)/\alpha)\bone{p_{k-1}(c)>\hp }]}{\log(\hp /\alpha)}\\
    &\leq \sum_{c\in L(v_k)}\frac{\E[p_{k-1}(c)\log p_{k-1}(c)]}{\log(\hp /\alpha)}-\sum_{c\in L(v_k)}\frac{\E[p_{k-1}(c)\log\alpha]}{\log(\hp /\alpha)}.
\end{align*}
In the first inequality, we used that fact that $\log (p_{k-1}(c) /\alpha) \geq \log (\hp /\alpha)$ whenever $p_{k-1}(c) > \hp$.
In the second inequality we used the fact that $p_{k-1}(c) \log (p_{k-1}(c) /\alpha)$ is a nonnegative random variable.
Indeed, at each iteration, a weight either increases, remains the same, or is set to $0$ and so, if $p_{k-1}(c) \neq 0$, then $p_{k - 1}(c) \geq \alpha$.
Applying Lemmas~\ref{lemma: P} and~\ref{lemma: H} to the above expression, we obtain 
\begin{align*}
\E\left[\sum_{c\in B_{k-1}(v_k)}p_{k-1}(c)\right] &\leq -\dfrac{1}{\log (\hp /\alpha)}\E[H_{k-1}(v_k)]-\frac{1}{\log(\hp /\alpha)}q \alpha \log\alpha\\
&\leq  \frac{1}{\log (\hp /\alpha)}\left(-H_0(v_k) + \frac{1}{1-\hp}\alpha q\left(\alpha d + \frac{2d^2\hp^2}{f}\right)\right)+\frac{1}{\log(\hp /\alpha)}H_0(v_k)\\
&=\alpha q\frac{1}{(1-\hp )\log (\hp /\alpha)}\left(\alpha d + \frac{2d^2\hp^2}{f}\right).
\end{align*}
It now suffices to argue that
\[\frac{1}{(1-\hp )\log (\hp /\alpha)}\left(\alpha d + \frac{2d^2\hp^2}{f}\right) \leq \frac{1}{2}.\]
To this end, we recall that $\alpha = 2\eta/(1-\eps/20)$, and $\eta$ and $\hp$ are defined in \eqref{eq: eta hp}.
It follows that
\[\alpha d + \frac{2d^2\hp^2}{f} = \frac{\log f}{(1-\eps/20)(4+ \eps/2)} + 2f^{-\frac{\eps}{20}} \leq \frac{\log f}{(1-\eps/18)(4+ \eps/2)},\]
and
\[(1-\hp )\log (\hp /\alpha) \geq (1-\hp )\log (f^{\frac{1}{2} - \frac{\eps}{45}}) \geq \frac{(1 - \eps/18)\log f}{2}.\]
Combining the above yields
\[\frac{1}{(1-\hp )\log (\hp /\alpha)}\left(\alpha d + \frac{d^2\hp^2}{f}\right) \leq \frac{2}{(1 - \eps/18)\log f}\frac{\log f}{(1-\eps/18)(4+ \eps/2)} \leq \frac{1}{2},\]
as desired.
\end{proof}

Plugging the bounds from Lemmas~\ref{lemma: P} and~\ref{lemma : B} into \eqref{eq: main}, we obtain:
\[\E[|S(v_k)]|=\E\left[P_{k-1}(v_k)\right]-\E\left[\sum_{c \in B_{k-1}(v_j)}p_{k-1}(c)\right] \geq \frac{\alpha q}{2},\]
completing the proof of Lemma~\ref{lemma: main1}.

\subsection{Proof of Lemma~\ref{lemma: main2}}\label{subsection: conc S}
In this section, we prove Lemma~\ref{lemma: main2}, showing that the size of each set \(S(v_k)\) is highly concentrated around its expected value. For the reader’s convenience, we restate the lemma below.

\begin{lemma*}[Restatement of Lemma~\ref{lemma: main2}]
\(\pr[|S(v_k)|\leq (1- \eps/20)\alpha q/2]=\exp\left( -\Omega\left( \dfrac{\eps^2q\alpha^2}{n^{2n}}\right)  \right)\).  
\end{lemma*}

As outlined in the proof overview, we aim to establish the concentration of \(|S(v_k)|\) using Theorem~\ref{theo: tal}. To this end, we define a collection of auxiliary random variables: for each \(c\in V(H)\), let \(a'_c\) be a independent random variable drawn uniformly from the interval \([0,1]\). We modify Algorithm~\ref{alg:dp-entropy} so that whenever \(c\in L(v_i)\), we set \(a_c= \bone{a'_c\leq p_{i-1}(c)}\). The distribution of this modified procedure is identical to that of the original algorithm. Since \(\cH\) is a \(q\)-fold cover of an \(n\)-vertex graph, \(|V(H)|=nq\); thus, \(|S(v_k)|\) is a function of the \(nq\) independent random variables \(\{a'_c\,:\,c\in V(H)\}\). 

We now bound the effect of altering a single trial $a'_c$. We claim that changing $a_c'$ affects the event $\bone{c' \in S}$ for at most $n^n$ colors $c'$.
Indeed, $a_c'$ can alter the outcome $\bone{c' \in S}$ if and only if there exists a path $(c = c_0,\, c_1, \ldots,\, c_t = c')$ in $H$ such that $c_{i+1} \in N_f(c_i)$ for each $0 \leq i < t$.
Because $|N_f(c_i)| \leq |N_f(L^{-1}(c_i))| \leq n$ and such ``strictly forward'' paths have length at most $n$, the claim follows.

By Lemma~\ref{lemma: main1}, \(\mathbb E[|S(v_k)|]\ge q\alpha/2.\)
Setting $t=\varepsilon q\alpha/16$, we obtain
\[
\left(1-\frac{\varepsilon}{8}\right)\frac{q\alpha}{2}
\leq \mathbb E[|S(v_k)|]-t.
\]
Therefore, applying Theorem~\ref{theo: tal} with parameters $s = nq$, $\xi = n^n$, and $t=\varepsilon q\alpha/20$ yields
\[
\pr\left[
|S(v_k)|\le
\left(1-\frac{\varepsilon}{8}\right)\frac{q\alpha}{2}
\right]
\le
\exp\left(
-\frac{\varepsilon^2q\alpha^2}{512\,n^{2n+1}}
\right)
=
\exp\left(
-\Omega\left(\frac{\varepsilon^2q\alpha^2}{n^{2n + 1}}\right)
\right),
\]
which completes the proof of Lemma~\ref{lemma: main2}.

\section{Concluding Remarks}\label{sec:concluding}

In this paper, we generalize Theorem~\ref{theo: martinsson steiner} in two key directions.
First, we extend the bound to the local sparsity setting while avoiding the large constant-factor loss inherent to the Alon--Krivelevich--Sudakov approach~\cite{AKSConjecture}.
Second, we establish this bound within the broader framework of DP-coloring, resolving a problem previously open even for triangle-free graphs.
Additionally, we provide a matching lower bound construction showing that both upper bounds are sharp up to their leading constants.
We conclude by highlighting several natural directions for future research.

In~\cite{martinsson2025random}, Martinsson and Steiner conjectured that the leading constant $4$ in Theorem~\ref{theo: martinsson steiner} can be reduced to $1$.
Allen, Noel, and the first author~\cite{high_girth} confirmed this conjecture under the stronger assumption that $G$ has girth at least $5$, providing a construction to show that this bound is tight (Theorem~\ref{theo: lower bound degen girth}).
For large $f$, this bound suggests that $2$ may be the correct constant in our setting.
Indeed, our lower bound construction for small $f$ has a very atypical neighborhood structure (disjoint union of cliques) which suggests it may not be optimal.
We conjecture that our lower bound can be improved to $2$ and that the analogous upper bound holds as well.

\begin{conjecture}\label{conj: local sparsity ub = lb}
    There exists a constant $C > 0$ such that the following holds for all sufficiently large $d$ and all $C \le f \le d^2 + 1$:
    \begin{enumerate}
        \item Every $d$-degenerate graph $G = (V, E)$ that admits a left $\frac{d^2}{f}$-locally-sparse degeneracy ordering satisfies $\chi_f(G) \le (2 + o(1))\frac{d}{\log f}$.
        \item There exists a $d$-degenerate graph $G = (V, E)$ that admits a left $\frac{d^2}{f}$-locally-sparse degeneracy ordering satisfying $\chi_f(G) \ge (2 - o(1))\frac{d}{\log f}$.
    \end{enumerate}
\end{conjecture}

For $K_{1, t, t}$-free graphs $G$, Anderson, Bernshteyn, and the first author showed that $\chi^{\mathrm{DP}}(G) \le (4 + o(1))\frac{\Delta}{\log \Delta}$, where $\Delta = \Delta(G)$ denotes the maximum degree.
Notably, this leading constant is independent of $t$.
It is natural to ask whether an analogous bound holds for fractional DP-coloring when maximum degree is replaced by degeneracy.

\begin{question}
    Does there exist an absolute constant $C > 0$ such that $\chiDP(G) \le C\frac{d}{\log d}$ for every $d$-degenerate $K_{1, t, t}$-free graph $G$?
\end{question}

A similar question arises for $K_{t, t, t}$-free graphs, where no non-trivial bound beyond the trivial estimate $\chi^{\mathrm{DP}}(G) \le \Delta(G) + 1$ is currently known.
Determining whether $\chi^{\mathrm{DP}}(G) = O_t\left(\frac{\Delta}{\log \Delta}\right)$ for all $K_{t, t, t}$-free graphs $G$ remains an open problem.

\subsection*{Acknowledgements}

We are grateful to Haoran Luo for helpful feedback on an earlier draft of this manuscript.

\printbibliography

\end{document}